\documentclass[12pt]{amsart}
\usepackage{amsfonts,amssymb,eucal}
\usepackage{color}

\usepackage{amsthm} 
 \usepackage{amsmath} 
 \usepackage{latexsym}
\usepackage{bbold,mathbbol,bbm}
\numberwithin{equation}{section}

\newcommand{\mc}[1]{{\mathcal{#1}}}
\newcommand{\got}[1]{{\mathfrak{#1}}}
\newcommand{\db}[1]{{\mathbb{#1}}}
\newcommand{\pa}{\partial}

\newcommand{\R}{\ensuremath{\mathbb{R}}}\newcommand{\C}{\ensuremath{\mathbb{C}}}
    \newcommand{\N}{\ensuremath{\mathbb{N}}}

\newcommand{\g}{\ensuremath{\got{g}}}
\newtheorem{Remark}{Remark}

\newtheorem{Theorem}{Theorem}
\newtheorem{Proposition}{Proposition}
\newtheorem{lemma}{Lemma}

\theoremstyle{definition} 
\def\ii{\operatorname{i}}
\newcommand{\Ka}{K\"ahler}
\newcommand{\mr}[1]{{\mathrm{#1}}}

\newcommand{\dd}{\operatorname{d}}

\newcommand{\h}{\ensuremath{\got{h}}}
\renewcommand{\Re}{\operatorname{Re}}
\renewcommand{\Im}{\operatorname{Im}}

\newcommand{\tr}{\operatorname{tr}}

\allowdisplaybreaks 
\newcommand{\Sp}{\ensuremath{{\mbox{\rm{Sp}}(n,\R)}}}

\begin{document}
\title[Connection matrices]{ Connection matrices on the Siegel-Jacobi upper
  half space and extended Siegel-Jacobi upper   half space}
\author{Elena Mirela Babalic}
\address[Elena Mirela Babalic]{``Horia Hulubei''  National
 Institute for Physics and Nuclear Engineering\\
         Department of Theoretical Physics\\
         PO BOX MG-6, Bucharest-Magurele, Romania}
\email{mbabalic@theory.nipne.ro}

\author{Stefan  Berceanu}
\address[Stefan  Berceanu]{``Horia Hulubei'' National
 Institute for Physics and Nuclear Engineering\\
         Department of Theoretical Physics\\
         PO BOX MG-6, Bucharest-Magurele, Romania}
       \email{Berceanu@theory.nipne.ro}

       \enlargethispage{1cm}
        \begin{abstract}
 The inverse of the  metric matrices  on  the  Siegel-Jacobi upper half
 space ${\mc{X}}^J_n$,   invariant
  to the  restricted real Jacobi group $G^J_n(\R)_0$ and extended  Siegel-Jacobi  $\tilde{{\mc{X}}}^J_n$
  upper half  space, invariant to the action  of the real Jacobi $G^J_n(\R)$, are  presented. The results are relevant for 
  Berezin  quantization of  the manifolds  ${\mc{X}}^J_ n$ and $\tilde{\mc{X}}^J_n$.
  Explicit calculations in  the case
  $n=2$ are given. 
\end{abstract}
\maketitle
\tableofcontents

\section{Introduction }

The real Jacobi group 
of  degree $n$ is defined as $G^J_n(\R):=\mr{H}_n(\R)\rtimes\Sp$, where
$\mr{H}_n(\R)$ denotes the   real Heisenberg group
\cite{bs,ez,tak}. 
The Siegel-Jacobi  upper half space is the  $G^J_n(\R)$-homogeneous
manifold 
$\mc{X}^J_n:=\frac{G^J_n(\R)}{\rm{U}(n)\times \R}\approx \mc{X}_n\times
\R^{2n}$, 
where  $\mc{X}_n$ denotes the Siegel upper half space realized as 
$ \frac{\Sp}{\rm{U}(n)}$ \cite[page 398]{helg}. Also $\mc{X}^J_n=
\frac{G^J_n(\R)_0}{\rm{U}(n)}$, where  $G^J_n(\R)_0$ denotes the
restricted  real Jacobi group \cite[\S 5]{nou}.

  The Jacobi
group 
$G^J_n:=\mr{H}_n\rtimes{\rm Sp}(n,\R)_{\C}$,
where ${\rm Sp}(n,\R)_{\C}:= {\rm Sp}(n,\C)\cap {\rm U}(n,n)$,
is  also studied in Mathematics,    Mathematical
Physics and Theoretical Physics,   together  with the 
 $G^J_n$-homogeneous Siegel-Jacobi ball $\mc{D}^J_n\approx
\C^n\times\mc{D}_n$,   
where $\mc{D}_n$ denotes the
Siegel ball realized as $\frac{{\rm Sp}(n,\C)}{{\rm U}(n)}$
\cite[page 399]{helg}.

The extended Siegel-Jacobi upper half space is  $\tilde{\mc{X}}^J_n\!:=\!\frac{G^J_n(\R)}{\rm{U}(n)}\approx \mc{X}_n\times
\R^{2n+1}$ \cite{SB21}.

 It
is well known that $G^J_n(\R)$, ${\rm Sp}(n,\R)$, $\mr{H}_n(\R)$,
$\mc{X}^J_n$
 and $\mc{X}_n$ are isomorphic with $G^J_n$, ${\rm Sp}(n,\R)_{\C}$, $\mr{H}_n$,
$\mc{D}^J_n$, 
respectively  $\mc{D}_n$ \cite{SB21}.    

The dimensions of the enumerated manifolds are: $\dim{\Sp}=2n^2+n$,
$\dim{\mr{H}_n(\R)}=2n+1$,
  $\dim{G^J_n(\R)}=(2n+1)(n+1)$,   $\dim{G^J_n(\R)_0}=\dim{G^J_n(\R)}-1$,
  $\dim{\mr{U}(n)}=n$,
  $\dim{\mc{X}^J_n}=n(n+3)$,
    $\dim{\tilde{\mc{X}}^J_n}= \dim{\mc{X}^J_n}+1$,
    $\dim{\mc{X}_n}=n(n+1)$. In  particular,  $\text{dim}~ \mc{X}^J_1 =4$,
     $\text{dim~} ~ \tilde{\mc{X}}^J_1 =5, ~\text{dim~} \mc{X}^J_2
    =10, ~\text{dim~} \tilde{\mc{X}}^J_1 =11$.

The Jacobi group,  as  a unimodular group, is a non-reductive, algebraic group of Harish-Chandra
type.  
$\mc{D}^J_n$ is a
 partially bounded domain, non-symmetric,  a  Lu Qi-Keng manifold,  a
 projectively induced quantizable \Ka~manifold \cite{SB15}.

 In our papers \cite{SB15,SB16,SB17,SB20,SB20N} 
we studied invariant metrics to the action of the Jacobi group and   Berezin quantisation
   \cite{ber73,ber74,ber75,berezin,arr,alo}  on  
   Siegel-Jacobi upper half-plane $\mc{X}^J_1$, extended
   Siegel-Jacobi upper half-plane
  $\tilde{\mc{X}}^J_1$, 
   Siegel-Jacobi space  $\mc{X}^J_n$,  Siegel-Jacobi ball
   $\mc{D}^J_n$.  One important step in the investigation of Berezin
   quantisation  on
   $\mc{X}^J_n$ and   $\tilde{\mc{X}}^J_n$ is the calculation of the 
   inverse matrix of invariant metric to the real  Jacobi group on the
   mentioned manifolds. The aim of the present investigation of
   the inverse of the invariant metric on $\tilde{\mc{X}}^J_n$. Even
   in case $n=2$ the formulas  presented here  are quite complicated.  

   The new results contained in this paper  concern the matrix of the inverse
metric on the Siegel-Jacobi space $g^{-1}_{\mc{X}^J_n}(x,y,p,q)$ in Proposition
\ref{PR1} and  extended Siegel-Jacobi  space $g^{-1}_{\tilde{\mc{X}}^J_n}(x,y,p,q,\kappa)$  in
Proposition \ref{PR2}, which are particularised in Proposition \ref{PR3} to the
case $n\!=\!2$. Also in Proposition \ref{PR5} we calculate  $\text{det}
g'_{\tilde{\mc{X}}^J_2}(x,y,q,p,\kappa)$ and the inverse of the 
same  matrix for the  case $n\!=\!2$. Remarks \ref{R1}  and \ref{R2}
 systematize 
the results of Proposition \ref{PR5}. 
 
\textbf{Notation.} 
We denote by $\mathbb{R}$, $\mathbb{C}$, $\mathbb{Z}$ and $\mathbb{N}$ 
 the field of real numbers, the field of complex numbers,
the ring of integers,   and the set of non-negative integers, respectively. We denote the imaginary unit
$\sqrt{-1}$ by~$\ii$, the real and imaginary parts of a complex
number $z\in\C$ by $\Re z$ and $\Im z$ respectively, and the complex 
conjugate of $z$ by $\bar{z}$.
 We denote by ${\dd }$ the differential. 
We use Einstein's summation convention, i.e.  repeated indices are
implicitly summed over.  The set of vector fields (1-forms) on real
manifolds  is denoted
by $\got{D}^1$ (respectively $\got{D}_1$). We denote a mixed tensor
contravariant of degree $r$ and covariant of degree $s$ by $\got{D}^r_s=\got{D}^r\times\got{D}_s$,
where $\got{D}^r=\underbrace{\got{D}^1\times \dots \times
  \got{D}^1}_{r}$ and $\got{D}_s=\underbrace{\got{D}_1\times \dots \times
  \got{D}_1}_{s}$ \cite[pages 13-17]{helg}. 
  If we
denote with Roman  capital letteres the Lie  groups, then their
associated Lie algebras are denoted with the corresponding lower-case
letteres.   If $\got{H}$ is a Hilbert space, then we adopt the
  convention  that the scalar product $(\cdot,\cdot)$ on
  $\got{H}\times\got{H}$ is antilinear in the first
factor 
$(\lambda a,b)=\bar{\lambda}(a,b),\quad \lambda \in \C\setminus \{0\}
$. If $\pi$ is a representation of a Lie grup $G$ 
on the Hilbert $\got{H}$  and $X\in\got{g}$,  then we denote
$\bf{X}:=\dd \pi (X)$   \cite{SB03,SB14,perG}. The interior  product $i_X\omega$ (interior multiplication or contraction)
of the differential form $\omega$
with $X\in\got{D}^1$ is denoted $X\lrcorner \omega$. We
denote by $M(n,m,\db{F})$ the set of $n\times m$ matrices with elements
in the field $\db{F}$ and  $M(n,\db{F})$ denotes $M(n,n,\db{F})$. If $X\in M(n,m,\db{F})$, then $X^t$ denotes the
transpose of $X$. We denote by $MS(n,\db{F})=\{X\in M(n,\db{F})|X=X^t\}$. The
 conjugate transpose (or hermitian transpose)  of $ A\in
 M(q,\C)$  is  $A^H:=\bar{A}^t$,  also denoted $A^*$, $A^{\dagger}$, $A^+$.
 If $f$ is a function on $\C^n$, we write the
total differential of $f$ as $\dd f\!=\! \pa f+\bar{\pa} f$, with $\pa f
\!=\!\sum_1^n{\pa_{\alpha}f}\dd z_{\alpha}$,
  where $\pa_{\alpha}f\!=\!\frac{\pa f}{\pa z_{\alpha}} $ \cite[page
  6]{GH}.
  If $F=F_{ij}\in M(n,\db F)$,  where $\db F=\R, ~\C$, we denote by $F^{ij}$
  the elements of the inverse $F^{-1}$,
  i.e. $F_{ik}F^{kj}=\delta_{ij}$. If $A,~B$ are matrices, then
   $A\otimes B$
  denotes the Kronecker product of $A,B$ \cite[page 3]{lutke}.

  \section{Preparation}\label{PR}
Following \cite[\S 5]{nou}, we consider the restricted real group
$G^J_n(\R)_0$ consisting of elements of the form $g=(M,X)$,
where $M\in \rm{Sp}(n,\R), X=(\lambda,\mu)$, $\lambda,~\mu\in M(1,n,\R)$.

We consider the
Siegel-Jacobi upper half space 
 $\mc{X}_n$ realized as in \eqref{uvpq}.

 We introduce  for $\mc{X}^J_n$ the analog
parametrization  used in \cite[page 7]{bern84}, \cite[page 11]{bs}, \cite[\S~38]{cal3} for  $\mc{X}^J_1$
\begin{equation}\label{uvpq}u:=pv+q, ~v:=x+\ii y,~v=v^t,~
  y>0,~p,q\in M(1,n,\R).\end{equation}

It should be noted that there is an isomorphism
$G^J_n(\R)\ni(M,X,K)\rightarrow(M,X)\in  G^J_n(\R)_0$ through which the
action of   $G^J_n(\R)_0$ on  $\mc{X}^J_n$ can be defined as in
\cite[Proposition\,\,\!2]{nou}.

    It is easy to prove  that   \cite[Lemma 8]{SB20N}
\begin{lemma}\label{actM}
a)  If $\mc{X}_n\ni v=x+\ii y$, then the action of $G^J_n(\R)_0$ on
$\mc{X}^J_n${\rm :}  $(M,X)\times
  (v',u')\rightarrow (v_1,u_1)$, where $M\in\Sp$  has the expression 
  \begin{equation}\label{Real}
M=\left(\begin{array}{cc}a&b\\c&d\end{array}\right),\quad a, b,c ,d \in
M(n,\R),
\end{equation}
is given by the formulas
 \begin{subequations}\label{exac}
    \begin{align}
      v_1 & =(av'+b)(cv'+d)^{-1}=(v'c^t+d^t)^{-1}(v'a^t+b^t),\label{exacc1}\\
      u_1 & = (u'+\lambda v'+\mu)(cv'+d)^{-1}. \label{exac2}
    \end{align}
  \end{subequations}

  b)  For $\lambda, \mu\in M(1,n,\R)$, let us consider $(p,q)$ such that 
  \begin{subequations}\label{pqlm}
    \begin{align}
  (p,q) & =(\lambda,\mu)M^{-1}=(\lambda d^t-\mu c^t, -\lambda b^t+ \mu a^t),\\
      (\lambda,\mu) & =(p,q)M=(pa+qc,pb+qd),
~~~ p,q,\lambda,\mu\in M(1,n,\R).
\end{align}
  \end{subequations} 
 Then  the action of $G^J_n(\R)_0$ on $\mc{X}^J_n${\rm :}   $(M,X)\times
  (x',y',p',q')\rightarrow (x_1,y_1,p_1,q_1)$ is given by \eqref{exacc1}, while
  \begin{equation}
     (p_1,q_1)=(p,q)+(p',q')\left(\begin{array}{cc}a & b \\c &
                                                          d\end{array}\right)^{-1}=
                                                      (p+p'd^t-q'c^t,q-p'b^t+q'a^t).\label{p1q1}
                                                            \end{equation}
 c) The action of $G^J_n(\R)$ on $\tilde{\mc{X}}^J_n\approx
 \mc{X}^J_n\times \R${\rm :}  
\begin{subequations}\label{actc}
\begin{align}
 (M,(\lambda,\mu),\kappa)\times (v',u',\kappa') &\rightarrow (v_1,u_1,\kappa_1),\\
(M,(\lambda,\mu),\kappa)\times
  (x',y',p',q',\kappa')& \rightarrow (x_1,y_1,p_1,q_1,\kappa_1)
\end{align}
\end{subequations}
is given by
  \eqref{exac}, \eqref{p1q1} and
  \[
    \kappa_1=\kappa+\kappa'+\lambda q'^t-\mu p'^t.
    \]
 d)  The $1$-form
  \begin{equation}\label{1FORMR}
    \lambda^R=\dd \kappa -p \dd q^t+q \dd p^t
  \end{equation}
  is invariant to the action \eqref{actc} of $G^J_n(\R)$ on
  $\tilde{\mc{X}}^J_n$.

  e) The action of $G^J_n(\R)$ on $G^J_n(\R)$ $$(M,(\lambda,\mu),\kappa)\times
  (S_n)'\rightarrow
  (S_n)_1,$$ is
  given in \eqref{exac1} for $X',Y'\in (S_n)',$
\begin{equation}\label{exac1}
    X_1-\ii Y_1=(y_1)^{\frac{1}{2}}\{(cx'+d)(y')^{-\frac{1}{2}}X'+c(y')^{\frac{1}{2}}Y'+
    \ii [c{y'}^{{\frac{1}{2}}}X'-(cx'+d)(y')^{-\frac{1}{2}}Y' ]\},
  \end{equation}
  while the other actions are
  given in {\emph{ a)-d)}}
  of the present Lemma.
\end{lemma}
  If parameters $\alpha,~\gamma, ~\delta >0$, we have proved \cite[Theorem 1]{SB20N}:
\begin{Theorem}\label{PRR2}
      The  metric on $\mc{X}^J_n$, $G^J_n(\R)_0$-invariant
    to the action in {\emph{Lemma \ref{actM}}}, has the expressions
\begin{subequations}\label{MULTL}\begin{align}
      \dd s_{\mc{X}^J_n}^2(x,y,q,p)&\!=\!\alpha\tr[(y^{-1}\dd
                                      x)^2+(y^{-1}\dd y)^2]\nonumber\\  
        &\!+\!\gamma[\dd p
        (xy^{-1}x+yy^{-1}y)\dd p^t+\dd q y^{-1}\dd q^t +2\dd p
        xy^{-1} \dd q^t];\\
  \dd s_{\mc{X}^J_n}^2(x,y,\chi,\psi)&\!=\!\alpha\tr[(y^{-1}\dd
                                      x)^2+(y^{-1}\dd y)^2]\nonumber\\  
        &\!+\!\gamma[\dd \psi^t
        (xy^{-1}x\!+\!yy^{-1}y)\dd\psi\!+\!\dd\chi^ty^{-1}\dd\chi\!+\!2\dd\psi^t
        xy^{-1} \dd\chi];\\
  \dd s_{\mc{X}^J_n}^2(x,y,\xi,\rho)&\!=\!\alpha\tr[(y^{-1}\dd
                                      x)^2+(y^{-1}\dd y)^2]\nonumber \\
  &\!+\!\gamma[ \dd \xi y^{-1}\dd \xi^t +\dd \rho y^{-1}\dd \rho^t
+\rho y^{-1}\dd x y^{-1}(\rho y^{-1}\dd x)^t  \nonumber\\&\!+\!\rho y^{-1}\dd y
    y^{-1}(\rho y^{-1}\dd y)^t  
 -2\rho y^{-1}\dd xy^{-1}\dd \xi^t -2\rho y^{-1}\dd y^{-1}\dd \rho^t].
      \end{align}
      \end{subequations}

    \noindent  The  three parameter metric on $\tilde{\mc{X}}^J_n$, $G^J_n(\R)$-invariant
      to the action {\emph{c)}} in {\emph{Lemma \ref{actM}}},
      is \begin{equation}\label{BIGM}\begin{split}
        \dd s_{\tilde{\mc{X}}^J_n}^2(x,y,q,p,\kappa)&=
        \dd s_{\mc{X}^J_n}^2(x,y,q,p)+ \lambda_6^2\\
         & =\alpha\tr[(y^{-1}\dd
                                      x)^2+(y^{-1}\dd y)^2]\\  
        & + \gamma[\dd p
        (xy^{-1}x+yy^{-1}y)\dd p^t+\dd q y^{-1}\dd q^t +2\dd p
        xy^{-1} \dd q^t]\\
        & +\delta (\dd \kappa -p \dd q^t+q \dd p^t)^2.
      \end{split}
      \end{equation}
      \end{Theorem}
   
        Formula \eqref{MULTL} ( \eqref{BIGM}) is a generalisation to
        $\mc{X}^J_n$ ( $\tilde{\mc{X}}^J_n$), $n\in \N$, of
        equation (5.25b) (respectively, (5.30)) in \cite{SB19} corresponding to $n=1$.

 \section{Berry phase on \Ka~ manifolds}\label{BEP}

 \subsection{Balanced metric}

 Let us denote \cite[page 138]{ccl}:
\[\Gamma_{ijk} :=g_{lj}\Gamma^l_{ik},\quad w_{ik}:=g_{lk}w^l_i.
\]
Also
\[\frac{\pa g_{ij}}{\pa u^k}=\Gamma_{ijk}+\Gamma_{jik}, 
\]or
\[
  g_{ij,k}:=\pa_kg_{ij}-\Gamma^l_{ki}g_{lj} -\Gamma^l_{kj}g_{il} =0,
\]
and
\begin{equation}\label{geot}
\Gamma_{ikj}\!=\!\frac{1}{2}\left( \frac{\pa g_{ik}}{\pa u^j} \!+ \!\frac{\pa g_{jk}}{\pa
    u^i}- \frac{\pa g_{ij}}{\pa u^k}\right), ~\Gamma^k_{ij}=\frac{1}{2}g^{kl}\left(\frac{\pa g_{il}}{\pa u^j}\!+\!\frac{\pa g_{jl}}{\pa
    u^i}\!-\!\frac{\pa g_{ij}}{\pa u^l}\right), ~ g_{ij}g^{jk}=\delta^k_i.
\end{equation}
$\Gamma_{ijk}$ ($\Gamma^i_{jk}$) is called Christofell's symbol of
first (respectively second) kind.

Let $v\!=\!x\!+\!\ii y$,  $x,y\in MS(n,\R)$, $y>0$,  $p,q\in M(1,n,\R)$.
Then (\cite[Theorem\,1]{SB22}, \cite[Proposition 3]{SB24}
\begin{equation}
   \dd s^2_{\mc{X}^J_n}(x,y,q,p)\! =\!\dd x g_{xx}\dd x^t\!+\!\dd yg_{yy}\dd
  y^t
  \!\!+\!\dd qg_{qq}\dd q^t\!\!+\!\dd p g_{pp}\dd p^t\!\!+\!\dd q g_{qp}\dd p^t\!\!+\!\dd p g_{pg}\dd q^t,
 \end{equation}
where
\begin{equation}\label{GG}
 g_{\mc{X}^J_n}(x,y,q,p)\!=\!\left(\begin{array}{cccc}
 g_{xx}&0&0&0\\
 0&g_{yy}&0&0\\
 0&0&g_{qq}&g_{qp}\\
 0&0&g_{pq}&g_{pp}\end{array}\right)
 \begin{array}{ll}
g_{xx}\!=\!g_{yy}\!=\!\alpha y^{-1}\otimes y^{-1}& g_{qq}\!=\!\gamma y^{-1} \\
 g_{pp}\!=\!\gamma(y\!+\!xy^{-1}x) & g_{pq}\!=\!\gamma xy^{-1} \\ g_{qp}\!=\!\gamma
   y^{-1}x & 
 \end{array} .
\end{equation}

\subsection{Partitioned inverses}
Let \begin{equation}
  h=\left(
    \begin{array}{cc}
    h_1 &h_2\\ h_3& h_4\end{array}
\right) ,
\end{equation}
where $h_1,h_2,h_3,h_4\in M(m,m), M(m,n), M(n,m), M(n,n)$. 

We have (see e.g. (1) and
(2) at p. 30 in \cite{lutke}): if  the matrices  $h_1$,
$h_4-h_3h^{-1}_1h_2$, $h_4$ and respectively  $h_1-h_2h^{-1}_4h_3$
are nonsingular, then
\begin{equation}\label{inv}
  h^{-1}=\left(\begin{array}{cc}h^1& h^2\\h^3 &
                                                h^4\end{array}\right)=\left( \begin{array}{cc}
                                                                               (h_1-h_2h^{-1}_4h_3)^{-1}&-h_1^{-1}h_2h^4\\- h^4 h_3h_1^{-1}& (h_4-h_3h^{-1}_1h_2)^{-1}\end{array}\right).\end{equation}

  If $A$ ($B$) are non singular matrices $ \in M(m,m)$, (respectively,
  $M(n,n)$), then  we
  have \cite[(11) (a) page 20]{lutke} the relation
  \[
    (A\otimes B)^{-1}=A^{-1}\otimes B^{-1}.
  \]
  \\
  
  If $A, ~B\in M(m,m)$  are  non-singular matrices, then  \cite[(5) page 29]{lutke}
  \[
   A^{-1}+B^{-1} =A^{-1}(A+B)B^{-1}.
    \]
    \begin{Proposition}\label{PR1}
      The inverse of the $MS(2n,\R)$-matrix \eqref{GG}  is
\begin{equation}\label{GG1}
 g_{\mc{X}^J_n}(x,y,q,p)^{-1}=\left(\begin{array}{cccc}
                                      \alpha^{-1}y\otimes y&0&0&0\\
0 &\alpha^{-1}y\otimes y&0 &0 \\
 0& 0 &\gamma^{-1}(y+xy^{-1}x)&-{\gamma}^{-1} xy^{-1}\\
 0&0&-{\gamma}^{-1} y^{-1}x& \gamma^{-1} y^{-1}\end{array}\right) .
\end{equation}
\end{Proposition}

\begin{proof}
  We apply \eqref{inv} to the matrix
\begin{equation}
\left(\begin{array}{cc} g_{qq}& g_{qp}\\ g_{pq}&g_{pp}\end{array}\right)=
                         \left(\begin{array}{cc}\gamma y^{-1} &\gamma y^{-1}x\\\gamma
          xy^{-1}&
        \gamma(y+xy^{-1}x)\end{array}\right).
  \end{equation}
  If we denote
  \begin{equation}
      \left(\begin{array}{cc} g^{qq} &g^{qp}\\
            g^{pq}&g^{pp}\end{array}\right):=\left(\begin{array}{cc}g_{qq}&g_{qp}\\g_{pq}&g_{pp}\end{array}\right)^{-1},
      \end{equation}
    then with formula \eqref{inv} we  get successively
  \begin{subequations}
    \begin{align}
      g^{qq}&=\left\{\gamma y^{-1}-\gamma y^{-1}x
              [\gamma (y+xy^{-1}x)]^{-1}\gamma
              xy^{-1}\right\}^{-1}\nonumber\\
            & =\gamma^{-1}\left\{y^{-1}- (x^{-1}y)^{-1} \left(y+xy^{-1}x\right)^{-1}(yx^{-1})^{-1} \right\}^{-1}\nonumber\\
  & =\gamma^{-1}\left\{y^{-1}-\left[yx^{-1}(y+xy^{-1}x)x^{-1}y\right]^{-1}\right\}^{-1}\nonumber\\
      &  =\gamma^{-1}\left\{y^{-1}-\left[yx^{-1}yx^{-1}y+yx^{-1}xy^{-1}xx^{-1}y\right]^{-1}\right\}^{-1}\nonumber\\ 
            &=\gamma^{-1}\left\{y^{-1}-\left[((yx^{-1})^2+1)y\right]^{-1}\right\}^{-1}\nonumber\\
      & =\gamma^{-1}\left\{y^{-1}-y^{-1}\left[1+(yx^{-1})^2\right]^{-1}\right\}^{-1}\nonumber\\
      &=\gamma^{-1}\left\{y^{-1}(1-[1+(yx^{-1})^2]^{-1} \right\}^{-1}.\nonumber
      \end{align}
    \end{subequations}
    Let us denote
    \[
      M:=(yx^{-1})^{2}.
    \]
    We apply formula \cite[(a) page 29]{lutke}
    \[
      (1+M)^{-1}=1-M+M^2-M^3+...~,
    \]
    and we get
    \begin{subequations}
      \begin{align}\label{39a}
      g^{qq}& =\gamma^{-1}\left\{y^{-1}M(1+M)^{-1}\right\} ^{-1} =\gamma^{-1}(y+xy^{-1}x).
       \end{align}
    \end{subequations} 
  \end{proof}
According with  \cite[Theorem 1]{SB22}, \cite[Proposition 3]{SB24}),  we consider now on the manifold $\tilde{\mc{X}}^J_n$ the metric matrix
\begin{equation}\label{GGP}
 g'_{\tilde{\mc{X}}^J_n}(x,y,q,p,\kappa)=\left(\begin{array}{ccc}
 g_{xx}&0&0\\
 0&g_{yy}&0\\
 0&0&g'(q,p,\kappa)
 \end{array}\right),\end{equation}
where
\begin{equation}\label{PQK}
  g'(q,p,\kappa)\!\!=\!\!
  \left(\!\begin{array}{ccc}
 g'_{qq}&g'_{qp}&g'_{q\kappa}\\
g'_{pq}& g'_{pp}&g'_{p\kappa}\\
g'_{\kappa q} &g'_{\kappa p}&g'_{\kappa\kappa}\end{array}\!\right)
 \begin{array}{ccc}
 g'_{qq}\!\!=\!\!g_{qq}\!+\!\delta p^t\!\otimes p&g'_{pp}\!=\!g_{pp}\!+\!\delta q^t\!\otimes q & g'_{\kappa\kappa}\!=\!\delta  \\
 g'_{pq}\!=\!g_{pq}\!-\!\delta p^t \otimes q& g'_{q\kappa}\!=\!-\!\delta p^t & g'_{p\kappa}\!\!=\!\!\delta q^t.
 \end{array}.\end{equation}
Note that if $p,q\in M(1,n,\db{F})$, then \cite[(12) page 20]{lutke}
\[
  p^t\otimes q=q^t\otimes p =p^tq =q^t p.
\]
We write \eqref{PQK} as
\begin{equation}
  g'(q,p,\kappa)=\left(\begin{array}{cc}A_{11}&A_{12}\\A_{21}&A_{22}\end{array}\right),
\end{equation}
where
\begin{subequations}\label{312}
  \begin{align}
  A_{11} & =\left(\begin{array}{cc}g'_{qq} & g'_{qp}\label{312a}\\
  g'_{pq}& g'_{pp}\end{array}\right)=
\left(\begin{array}{cc}\gamma y^{-1}+\delta p\otimes p^t& \gamma
                                                   y^{-1}x-\delta
                                                   p\otimes
                                                   q^t \\ \gamma
        xy^{-1}-\delta q\otimes p^t & \gamma(y+xy^{-1}x)+\delta q\otimes
                                      q^t\end{array}\right),
                                \\
                                  ~A_{12}& =\left(\begin{array}{l}g'_{q\kappa}\\g'_{p\kappa}\end{array}\right)
                                  =\left(\begin{array}{c}-\delta p^t\\
                                           \delta q^t\end{array}\right),~
 A_{21}=\left(\begin{array}{cc}g'_{\kappa q}~~g'_{\kappa p}
       \end{array}\right)=\left(\begin{array}{cc}-\delta p ~~\delta q  \end{array}\right),~A_{22}=g'_{\kappa \kappa}.
\end{align}
   \end{subequations}

 Now we calculate the inverse matrix of \eqref{GGP}  
\begin{equation}\label{GGPI}
 g'_{\tilde{\mc{X}}^J_n}(x,y,q,p,\kappa)^{-1}=\left(\begin{array}{ccc}
                                      \alpha^{-1}y\otimes y&0&0\\
0 &\alpha^{-1}y\otimes y&0  \\
 0& 0 & g'^{-1}(q,p,\kappa)
 \end{array}\right) ,
\end{equation}
where
 \begin{equation}\label{316}
       g'^{-1}(q,p,\kappa)=
       \left(\begin{array}{cc}b_{11}&b_{12}\\b_{21}& b_{22}\end{array}\right).
       \end{equation}

   With  the notation 
   \begin{equation}\label{A11}
     A^{-1}_{11}=\left(\begin{array}{cc}a_{11}& a_{12}\\a_{21}& a_{22}\end{array}\right),
   \end{equation}
 the values of the matrix \eqref{A11}  became with \eqref{inv}
   \begin{subequations}\label{XYY}
     \begin{align}
       a_{11} & \!\!=\!\![\gamma y^{-1}\!\!+\!\!\delta p^t\otimes
                p\!\!-\!\!( \gamma y^{-1}x\!\!-\!\!\delta 
     q^t\otimes p)\\& \times (\gamma (y\!\!+\!\!xy^{-1}x)\!\!+\!\!\delta q^t\otimes q)^{-1}(\gamma xy^{-1}\!\!-\!\!\delta
               p^t\otimes q)]^{-1},\nonumber\\
       a_{12} &\!\!=\!\!-(\gamma y^{-1}+\delta p^t\otimes p)^{-1}(
                \gamma y^{-1}x-\delta
                q^t\otimes p)a_{22},\\
       a_{21} & \!\!=\!\!-a_{22}( \gamma xy^{-1}-\delta
       p^t\otimes q)(\gamma y^{-1}+\delta p^t\otimes p)^{-1}\\
       a_{22} &\!\!=\!\![\gamma(y\!\!+\!\!xy^{-1}x)\!\!+\!\!\delta
                q^t\otimes q\!\!-\!\!(\gamma  xy^{-1}\!\!-\!\!\delta
                p^t\otimes q)\nonumber\\ &  \times (\gamma y^{-1}\!\!+\!\!\delta p^t\otimes
                p)^{-1}( \gamma y^{-1}x\!\!-\!\!\delta q^t\otimes p)]^{-1}\label{317e}.
       \end{align}
     \end{subequations}
     \begin{Proposition}\label{PR2}
       For $a_{22}$ appearing in \eqref{A11}, \eqref{317e} we find the
       simplified expression
\begin{subequations}\label{FINAL}
     \begin{align}
       a_{22} &= [\gamma y +\delta q^t\otimes q+\delta x p^t\otimes
                py(1+\frac{\delta}{\gamma}p^t\otimes py)^{-1}y^{-1}x \nonumber\\
       & +\gamma\delta x(\gamma +\delta p^t\otimes py)^{-1}q^t\otimes
      p\nonumber\\
    &~~~+\gamma\delta p^t\otimes q (\gamma +\delta y p^t\otimes
      p)^{-1} x\nonumber\\
    &~~~-\delta^2p^t\otimes q y(\gamma +\delta p^t\otimes p
      y)^{-1}q^t\otimes p]^{-1}\nonumber .
     \end{align}
     \end{subequations}
For the matrix elements in \eqref{316} we find 
  \begin{subequations}\label{3166}
         \begin{align}
           b_{11}&=(A_{11}-A_{12}\delta^{-1}A_{21})^{-1},\label{316a}\\
           b_{12}&=-A_{11}^{-1}A_{12}b_{22},\label{316b}\\
           b_{21}&=- b_{22}A_{21}A_{11}^{-1},\label{316c}\\
           b_{22}&=(\delta -A_{21}A_{11}^{-1}A_{12})^{-1},\label{316d}
           \end{align}
         \end{subequations}
     and
         \[
           b_{11}=\left(\begin{array}{cc}g^{qq}&g^{qp}\\g^{pq}&g^{pp}\end{array}\right),
         \]
         \[
           b_{22}=\delta^{-1}[1-\delta(pa_{11}p^t-qa_{21}p^t-pa_{12}q^t+qa_{22}q^t)]^{-1}.
         \]
       \end{Proposition}

     \begin{proof}
     We write \eqref{317e} as
     \begin{equation}
     a_{22} = \![\gamma(y\!\!+\!\!xy^{-1}x)\!\!+\!\!\delta
     q^t\otimes q+ XX]^{-1},
   \end{equation}
   where
   \begin{equation}XX := -(\gamma  xy^{-1}\!\!-\!\!\delta
                p^t\otimes q)  (\gamma y^{-1}\!\!+\!\!\delta p^t\otimes
         p)^{-1}( \gamma y^{-1}x\!\!-\!\!\delta q^t\otimes p).
       \end{equation}
We have successively
\begin{subequations}\label{XXX}
  \begin{align}
    XX& = -\gamma xy^{-1}(\gamma y^{-1}\!\!+\!\!\delta p^t\otimes
         p)^{-1}\gamma y^{-1}x\nonumber\\
       & ~~~+\gamma xy^{-1}(\gamma y^{-1}\!\!+\!\!\delta p^t\otimes
         p)^{-1}\delta q^t\otimes p\nonumber\\
    &~~~+\delta p^t\otimes q (\gamma y^{-1}\!\!+\!\!\delta p^t\otimes
      p)^{-1}\gamma y^{-1}x\nonumber\\
    & ~~~-\delta p^t\otimes q (\gamma y^{-1}\!\!+\!\!\delta p^t\otimes
      p)^{-1}\delta q^t\otimes p \nonumber\\
    & =-\gamma^2(yx^{-1})^{-1}(\gamma y^{-1}\!\!+\!\!\delta p^t\otimes
      p)^{-1}(x^{-1}y)^{-1}\nonumber\\
    & ~~~+\gamma \delta (yx^{-1})^{-1}(\gamma y^{-1}\!\!+\!\!\delta p^t\otimes
      p)^{-1} q^t\otimes p\nonumber\\
    &~~~ +\gamma \delta p^t\otimes q (\gamma y^{-1}\!\!+\!\!\delta p^t\otimes
    p)^{-1}(x^{-1}y)^{-1}\nonumber\\
    & ~~~-\delta^2 p^t\otimes q (\gamma y^{-1}\!\!+\!\!\delta p^t\otimes
      p)^{-1}q^t\otimes p\nonumber\\
   & = -\gamma^2[x^{-1}y(\gamma y^{-1}+\delta p^t\otimes
     p)yx^{-1}]^{-1}\nonumber\\
    &~~~+\gamma \delta [(\gamma y^{-1}\!\!+\!\!\delta p^t\otimes
      p) yx^{-1}]^{-1}q^t\otimes p\nonumber\\
    &~~~+\gamma \delta  p^t\otimes q [(x^{-1}y)(\gamma y^{-1}+\delta p^t\otimes
      p)]^{-1}\nonumber\\
      &~~~-\delta^2p^t\otimes q (\gamma y^{-1}+\delta p^t\otimes
        p)^{-1}q^t\otimes p\nonumber\\
    &= -\gamma^2[\gamma x^{-1}yx^{-1}+\delta x^{-1}y p^t\otimes
      pyx^{-1}]^{-1}\nonumber\\
   & +~~~\gamma \delta (\gamma x^{-1}+\delta p^t\otimes p
    yx^{-1})^{-1}q^t\otimes p\nonumber\\
    &+~~~\gamma \delta p^t\otimes q (\gamma x^{-1}+\delta x^{-1}y p^t\otimes p)^{-1}\nonumber\\
     & ~~~-\delta^2 p^t\otimes q (\gamma y^{-1}\!\!+\!\!\delta p^t\otimes
       p)^{-1}q^t\otimes p\nonumber\\
      &=-\gamma^2[x^{-1}y(\gamma +\delta p^t\otimes py)x^{-1}]^{-1}\nonumber\\
    &~~~+\gamma\delta \left[(\gamma +\delta p^t\otimes
      py)x^{-1}\right]^{-1}q^t\otimes p\nonumber\\
    &~~~+\gamma\delta p^t\otimes q[x^{-1}(\gamma +\delta y p^t\otimes
      p)]^{-1}\nonumber\\
    & -\delta^2 p^t\otimes q (\gamma y^{-1}\!\!+\!\!\delta p^t\otimes
      p y y^{-1})^{-1}q^t\otimes p\nonumber\\
     &=-\gamma^2x(\gamma +\delta p^t\otimes
        py)^{-1}y^{-1}x+XY\nonumber,
      \end{align}
    \end{subequations}
    where
    \begin{subequations}
      \begin{align}
    XY&:=\gamma\delta x(\gamma +\delta p^t\otimes py)^{-1}q^t\otimes
      p\nonumber\\
    &~~~+\gamma\delta p^t\otimes q (\gamma +\delta y p^t\otimes
      p)^{-1}x \nonumber\\
    &~~~-\delta^2p^t\otimes q y(\gamma +\delta p^t\otimes p
      y)^{-1}q^t\otimes p\nonumber .
    \end{align}
  \end{subequations}
  We have
  \begin{equation}\label{A22F}
    a_{22}= [XZ+XY]^{-1},
    \end{equation}
  where
 \begin{subequations}
    \begin{align}
    XZ&:= \gamma (y+xy^{-1}x)+\delta q^t\otimes q-
\gamma^2x(\gamma +\delta p^ty)^{-1}y^{-1}x\nonumber\\
 &= \gamma (y+xy^{-1}x)+\delta q^t\otimes q-\gamma
 x(1+A)^{-1}y^{-1}x,  A:=\frac{\delta}{\gamma}p^t\otimes py
   \nonumber\\
      & = \gamma y+\gamma xy^{-1}x+\delta q^t\otimes q -\gamma
        x(1-A+A^2-...)y^{-1}x\nonumber\\
      & =\gamma y+\delta q^t\otimes q+\gamma xA(1+A)^{-1}y^{-1}x, 
 \end{align}
 \end{subequations}
 and  we find
 \begin{equation}\label{323X}
   XZ=\gamma y +\delta q^t\otimes q+\delta x p^t\otimes
   py(1+\frac{\delta}{\gamma}p^t\otimes py)^{-1}y^{-1}x.
   \end{equation}
 With \eqref{317e},  \eqref{A22F}, \eqref{323X},  we  finally find for
 $a_{22}$ the expression \eqref{FINAL}.
\end{proof}


For ${\mc{X}}^J_1$  we have
\begin{equation}\label{322}
  \left(\begin{array}{cc}g_{qq}&g_{qp}\\g_{pq}&g_{pp}\end{array}\right)^{-1}=[\gamma\left(\begin{array}{cc}\frac{1}{y}&\frac{x}{y}\\\frac{x}{y}& \frac{x^2+y^2}{y}\end{array}\right)]^{-1}=
  \frac{1}{\gamma y}\left(\begin{array}{cc}x^2+y^2&-x\\-x&
                                                           1\end{array}\right),
                                                       \end{equation}
                                                       which is
                                                       compatible with
                                                       \eqref{GG1}.
                                                       
For $\tilde{\mc{X}}^J_1$ we have 
\begin{equation}\label{323}
  \left(\begin{array}{cc}g'_{qq}&g'_{qp}\\g'_{pq}&g'_{pp}\end{array}\right)^{-1}\!=\!\frac{1}{\gamma
 }\frac{1}{\gamma
  y\!+\!\delta[(px+q)^2\!+\!y^2
  p^2]}\left(\begin{array}{cc}\gamma(x^2\!+\!y^2)\!+\!\delta q^2y& -\gamma
                                                           x\!+\!\delta p
                                                           q y\\\!-\!\gamma
                                                           x\!+\!\delta p
                                                           q y &
                                                                 \gamma
                                                                 \!+\!\delta
                                                                 p^2 y\end{array}
                                                            \right) . \end{equation}
                                                           If in
                                                           \eqref{323}
                                                           we take
                                                           $\delta\rightarrow 0$
                                                           we get \eqref{322}.

Let $(M,g)$ be a $n$-dimensional  Riemannian manifold. In  a local coordinate system $x^1,\dots$, $x^n$ the geodesic equations on a manifold $M$ 
with components of the linear connection $\Gamma$ are (see e.g.  \cite[Proposition 7.8 p 145]{kn1}):
\begin{equation} \label{GEOO}
\frac{\dd ^2 x^i}{\dd t^2}  +\sum_{j,k=1}^n\Gamma^i_{jk}\frac{\dd x^j}{\dd t}
\frac{\dd x^k}{\dd t}  =  0, ~i=1,\dots,n .
\end{equation}
The components $\Gamma^i_{jk}$ (Christofell's symbols of second kind) of a Riemannian (Levi-Civita)
connection $\nabla$  
are obtained from the  Christofell's symbols of first kind $[ij,k]$ by solving  the linear system of algebraic equations, see e.g. \cite[p 160]{kn1}
\begin{equation}\label{geot1}g_{lk}\Gamma^l_{ji}=[ij,k]=:\frac{1}{2}\left(\frac{\pa
      g_{ki}}{\pa x_j}+\frac{\pa g_{jk}}{\pa x_i}-\frac{\pa g
      _{ji}}{\pa x_k}\right),~~i,j,k,l=1,\dots,n. \end{equation}
{\it The} $\frac{n^2(n+1)}{2}$ {\it distinct} $\Gamma$-{\it symbols of an}
 $n$-{\it dimensional  manifold are given by the formulas}
\begin{equation}\label{geoI}
  \Gamma^m_{ij}=[ij,k]g^{km}, \quad
\text{where~ ~ ~ }g^{mk}g_{kl}=\delta ^m_l, ~~i,j,k,m=1,\dots,n.
\end{equation}

Applying \eqref{geot1}, \eqref{geoI} to \eqref{322} (\eqref{323}) for
$\mc{X}^J_1$   (respectively $\tilde{\mc{X}}^J_1$), we  get  again
$\Gamma$-s obtained in \cite[(22)]{SB21}  (respectively,  \cite[page
29]{SB21}).

We apply \eqref{geoI} to calculate $\Gamma^{x}_{xy}$
\[
  \Gamma^x_{xy}=[xy,k]g^{kx}=\frac{1}{2}(\frac{\pa g_{xx}}{\pa
    y}+\frac{\pa g_{yx}}{\pa x}-\frac{\pa g_{xy}}{\pa x^k})g^{kx}.
\]
For example,  for \eqref{GGPI}, \eqref{geot} gives
\[
  \Gamma^x_{xy}=\frac{1}{2}(\frac{\pa g_{xx}}{\pa
    y}+\frac{\pa g_{yx}}{\pa x}-\frac{\pa g_{xy}}{\pa x})g^{xx}=-\frac{1}{y},
  \]
  i.e. the first term in \cite[(73)]{SB21}.

  Similarly, with  \eqref{312},  we have on $\tilde{\mc{X}}^J_1$
 \begin{equation}
   g'(q,p,\kappa)=\left(\begin{array}{ccc}\gamma y^{-1}+\delta p^2&
                                                                    \gamma
                                                                    y^{-1}x-\delta
                                                                    pq
                          & -\delta p\\\gamma xy^{-1} -\delta p q&
                                                                   \gamma
                                                                   (y+xy^{-1}x)+
                                                                   \delta
                                                                   q^2&
                                                                        \delta
                                                                        q\\
                          -\delta p&\delta q & \delta\end{array}\right) .
                    \end{equation}
                    We find
                    \begin{equation}
                      \text{det} g'(x,y,q,p,\kappa)=\gamma^2 \delta.
                      \end{equation}
                      We now calculate $\Gamma^p_{pp}$
                      \begin{equation}\label{SUMM}
\Gamma^p_{pp}=[pp,x]g^{xp}+[pp,y]g^{yp}+[pp,q]g^{qp}+[pp,p]g^{pp}+[pp,\kappa]g^{\kappa p}.
                      \end{equation}
  With formula \eqref{geot1}, we find that the only non-zero
  Chistofell's symbol in \eqref{SUMM} is
  \[
[pp,q]=-2 \delta q
\]
and
\[
  \Gamma^{p}_{pp}=-2\delta q g^{qp},
\]
where
\[
  g^{qp}=-\frac{1}{\gamma^2\delta}\text{det}\left(\begin{array}{cc}\gamma
                                                    xy^{-1}-\delta pq&
                                                                       \delta
                                                                       q\\
                                                    -\delta p&
                                                               \delta\end{array}\right)=-\frac{1}{\gamma} xy^{-1},
  \]
   \[
\Gamma^p_{pp}=2 \frac{\delta}{\gamma}xy^{-1}q,
  \]
and we  regain  the value calculated in \cite[page 22]{SB21}.
  
  The following  Proposition is \cite[Proposition 5.6]{SB19} for $\tilde{\mc{X}}^J_1$
         \begin{Proposition}\label{Pr5}The metric on the extended Siegel--Jacobi upper half-plane $\tilde{\mc{X}}^J_1$, in the partial S-coordinates $(x,y,p,q,\kappa)$:
\begin{gather}{\rm d} s^2_{\tilde{\mc{X}}^J_1} ={\rm d}
 s^2_{\mc{X}^J_1}(x,y,p,q)+\lambda^2_6(p,q,\kappa)\nonumber\\
\hphantom{{\rm d} s^2_{\tilde{\mc{X}}^J_1}}{} =\frac{\alpha}{y^2}\big({\rm d} x^2+{\rm d}
 y^2\big)+\left[\frac{\gamma}{y}\big(x^2+y^2\big)+\delta q^2\right]{\rm d} p^2+
 \left(\frac{\gamma}{y}+\delta p^2\right){\rm d} q^2 +\delta {\rm d} \kappa^2\nonumber\\
\hphantom{{\rm d} s^2_{\tilde{\mc{X}}^J_1}=}{} + 2\left(\gamma\frac{x}{y}-\delta pq\right){\rm d} p{\rm d} q +2\delta (q{\rm d} p{\rm d}
 \kappa-p{\rm d} q {\rm d} \kappa)\label{linvG}
 \end{gather}
 is left-invariant with respect to the action given in \cite[Lemma 5.1]{SB19} of the Jacobi group $G^J_1(\R)$.

The matrix attached to metric \eqref{linvG} is
\begin{gather}\label{335}
g_{\tilde{\mc{X}}^J_1} = \left(\begin{matrix}g_{xx} &0 &0 &0&0\\
0& g_{yy}& 0& 0 & 0\\
0& 0& g'_{qq} & g'_{qp} & g'_{q\kappa}\\0 & 0& g'_{pq}& g'_{pp}
 &g'_{p\kappa}\\
0& 0& g'_{\kappa q}& g'_{\kappa p} & g'_{\kappa\kappa}
\end{matrix}\right),\nonumber\\
g'_{pq} = g_{pq}-\delta p q , \qquad g'_{p\kappa}=\delta q, \qquad g'_{q\kappa}=-\delta p, \qquad
g'_{pp} = g_{pp}+\delta q^2,\nonumber\\
g'_{qq} = g_{qq}+\delta p^2, \qquad g'_{\kappa\kappa} = \delta.\label{begGG}
\end{gather}
We have 
\begin{equation}
g_{\mc{X}^J_1} = \left(\begin{matrix}g_{xx} &0 &0 &0\\
0& g_{yy}& 0& 0\\
0& 0& g_{qq} & g_{qp} \\0 & 0& g_{pq}& g_{pp}
\end{matrix}\right),
~\begin{array}{l}g_{xx}=g_{yy}=\frac{\alpha}{y^2}\\
   g_{qq}=\frac{\gamma}{y},~
g_{pp}=\gamma\frac{x^2+y^2}{y},~
   g_{pq}=\gamma \frac{x}{y}
   \end{array}
\end{equation} associated with the balanced metric $${\rm d} s^2_{\mc{X}^J_1}(x,y,p,q) =
c_1\frac{{\rm d} x^2 + {\rm d} y^2}{4y^2}
+\frac{c_2}{y}\big[\big(x^2+y^2\big){\rm d} p^2+{\rm d} q^2+2x{\rm d}
p{\rm d} q\big] $$on~$\mc{X}^J_1$, where $\frac{c_1}{4}=\alpha, c_2=\gamma$.

The extended Siegel--Jacobi upper half-plane $\tilde{\mc{X}}^J_1$ does not admit an almost contact structure $(\Phi,\xi,\eta)$ with a contact form $\eta=\lambda_6$ and Reeb vector $\xi= \operatorname{Ker}(\eta)$.
\end{Proposition}

 \subsection{The inverse of the invariant metric matrix on $\tilde{\mc{X}}^J_2 $} 

We particularize formula \eqref{312a} for the case $n=2$ and we
have
\begin{Proposition}\label{PR3}  
\begin{equation*}
   A_{11}=\left(\begin{array}{cc}  \left(\begin{array}{c}\gamma\left(\begin{array} {cc}y_{11}&
                                                                       y_{12}\\y_{21} &y_{22}
                                                             \end{array}\right)^{-1}
                                                                                        \\
                                                                                        +\delta
                                                                                        \left(\begin{array}{c}p_1\\
                                                                                                p_2\end{array}\right)\left(\begin{array}{cc}
                                                                                                                   p_1&p_2\end{array}\right)\end{array}\right)&
                                                                                                                                                                 \left(\begin{array}{c}\gamma
 \left(\begin{array}{cc}y_{11}&y_{12}\\y_{21}&y_{22}\end{array}\right)^{-1}\left(\begin{array}{cc}x_{11}&x_{12}\\x_{21}
 &x_{22}\end{array}\right)\\-\delta\left(
   \begin{array}{c}q_1\\q_2\end{array}\right)\left(\begin{array}{cc}p_1&p_2\end{array}\right)\end{array}\right)
\\
   \left(\begin{array}{c}         \gamma
           \left(\begin{array}{cc}x_{11}&x_{12}\\x_{21}&
                                                         x_{22}\end{array}\right)
                                                        \left(\begin{array}{cc}y_{11}&y_{12}\\y_{21}&
                                                                                                      y_{22}\end{array}\right)^{-1}   \\
                  -\delta\left(\begin{array}{c}p_1\\p_2\end{array}\right)
     \left(\begin{array}{cc}q_1&q_2\end{array}\right)\end{array}\right)&
                                                                         \left( \begin{array}{c}\gamma  \begin{array}{c} \left( \begin{array}{cc}\!y_{11}\!&y_{12}\!\\\!y_{21}\!&\!y_{22}\!\end{array}\right)\end{array} \!\!+\!\!
                  \gamma\left(\begin{array}{cc}\!x_{11}\!&\!x_{12}\!\\\!x_{21}\!&\!x_{22}\!\end{array}\right)\\
                                                                                  \times\left(\begin{array}{cc}y_{11}&y_{12}\\y_{21}&y_{22}\end{array}\right)^{-1}\left(\begin{array}{cc}x_{11}&x_{12}\\x_{21}&x_{22}\end{array}\right)
                                                                                                                                                                                                                
                                                                                \\+\delta
                                                                                  \left(\begin{array}{c}q_1\\q_2\end{array}\right)\left(\begin{array}{cc}q_1&
                                                                                                                                                              q_2\end{array}\right)\end{array}\right)\end{array}\right)
\end{equation*}
 where                                                                       
\[ y^{-1}=
  \left(\begin{array}{cc}y_{11} & y_{12}\\ y_{21} &
                                                    y_{22}\end{array}\right)^{-1}=\frac{\left(\begin{array}{cc}y_{22}
                                                                                                &-y_{12}\\-y_{21}&
                                                                                                                   y_{11}\end{array}\right)}{D},~~ D:=y_{11}y_{22}-y_{12}y_{21}.
 \]
Particularizing \eqref{XYY} for $n=2$ we get
 \begin{subequations}
   \begin{align}
   a_{11} &=[\gamma y^{-1}\!+\!\delta\left(\begin{array}{c}p_1\\p_2
                                     \end{array}\right)\left(\begin{array}{cc}p_1&p_2\end{array}\right)-(\gamma
                                   y^{-1}x-\delta
                                   \left(\begin{array}{c}q_1\\q_2\end{array}\right)\otimes
                                           \left(\begin{array}{cc}p_1&p_2\end{array}\right))\times\\
     &\times (\gamma( y\!\!+\!\!xy^{-1}x)
     \!\!+\!\!\delta \left(\begin{array}{c}q_1\\q_2\end{array}\right)\otimes\left(\begin{array}{cc}q_1&q_2\end{array}\right))^{-1}(\gamma
                                                     xy^{-1}\!\!-\!\!\delta \left(\begin{array}{c}p_1\\p_2\end{array}\right)\left(\begin{array}{cc}q_1&q_2\end{array}\right)]^{-1}
   \\
a_{22}= &  [\gamma (y\!\!+\!\!x\!y^{-1}x) \!\!+\!\!\delta
          \left(\begin{array}{c}q_1\\q_2\end{array}\right)\otimes
     \left(\begin{array}{cc}q_1&q_2\end{array}\right)\! \!-\!\!(\gamma xy^{-1}\!\!-\!\!\delta \left(\begin{array}{c}p_1\\p_2\end{array}\right)
                                       \otimes\left(\begin{array}{cc}
                                                q_1&q_2\end{array}\right))
                                 \nonumber\\
     & \times (\gamma
                                                     y^{-1}\!+\!\delta
                                                     \left(\begin{array}{c}p_1\\p_2\end{array}\right)\otimes\left(\begin{array}{cc}p_1&p_2\end{array}\right))^{-1}(\gamma
                                                                                                                                 y^{-1}x\!\!-\!\!\delta\left(\begin{array}{c}
                                                                                                                                 q_1\\
                                                                                                                                 q_2\end{array}\right)\otimes\left(\begin{array}{cc}
                                                                                                                                 p_1&
                                                                                                                                      p_2\end{array}\right))]^{-1}\\
     a_{12} & =-(\gamma y^{-1}\!\!+\!\!\delta
     \left(\begin{array}{c}p_1\\p_2\end{array}\right)\otimes\left(\begin{array}{cc}p_1&p_2\end{array}\right))^{-1}(\gamma
                                                                                        y^{-1}x\!\!-\!\!\delta
                                                                                        \left(\begin{array}{c}q_1\\q_2\end{array}\right)\otimes
     \left(\begin{array}{cc}p_1&p_2\end{array}\right))a_{22}\\
    a_{21}&= -a_{22}(\gamma
    xy^{-1}\!-\!\delta\left(\begin{array}{c}p_1\\p_2\end{array}\right)\otimes
    \left(\begin{array}{cc}q_1&q_2\end{array}\right))(\gamma
    y^{-1}\!\!+\!\!\delta
   \left(\begin{array}{c}p_1\\p_2\end{array}\right)\otimes\left(\begin{array}{cc}p_1&p_2\end{array}\right))^{-1}
   \end{align}
  \end{subequations}
  With \eqref{316}, \eqref{312}  and \eqref{3166},  we get for $n=2$ 
    \begin{subequations}
      \begin{align}
    b_{11}\gamma &= \left(\begin{array}{cc}y^{-1} & y^{-1}x\\
      xy^{-1} &  y+xy^{-1}x\end{array}\right)^{-1}=
    \left(\begin{array}{cc}(y+xy^{-1}x)^{-1} &-yx^{-1}\\-xy^{-1}& y^{-1}\end{array}\right)\\
  b_{22}\delta &=\left[1-\delta \left(\begin{array}{cc}-p& q\end{array}\right)\left(\begin{array}{cc}a_{11}&a_{12}
  \\a_{21}&a_{22}\end{array}\right)\left(\begin{array}{c}-p^t\\
                                           q^t\end{array}\right)\right]^{-1}\nonumber\\
   & = \left[1-\delta(pa_{11}p^t-qa_{21}p^t-pa_{12}q^t+qa_{22}q^t) \right]^{-1}\\
\frac{b_{12}}{-\delta}& =\left(\begin{array}{c}
    -a_{11}\left(\begin{array}{c}p_1\\p_2\end{array}
        \right)+a_{12}\left(\begin{array}{c}q_1\\q_2\end{array}\right)\\
        -a_{21}\left(\begin{array}{c}p_1\\p_2\end{array}\right)+a_{22}\left(\begin{array}{c}q_1\\q_2\end{array}\right)\end{array}\right) b_{22}\\
  \frac{b_{21}}{-\delta}&=b_{22}\left(-\left(\begin{array}{cc}p_1\!&\!p_2\end{array}\right)
                                                                   a_{11}+\left(\begin{array}{cc}q_1\!&\!q_2\end{array}\right)
                                                                                                      a_{21},-\left(\begin{array}{cc}p_1\!&\!p_2\end{array}\right)
                                                                                                                                    a_{12}\!+\!\left(\begin{array}{cc}q_1\!&\!q_2\end{array}\right)a_{22}\right)
      \end{align}
      \end{subequations}
    \end{Proposition}

Now we particularize \eqref{PQK} for $\tilde{\mc{X}}^J_2 $ 
\begin{equation}\label{mdet}
g' _{ij}=g'_{\tilde{\mc{X}}^J_2}(x,y,q,p,\kappa), ~~ i,j =1,2,
\end{equation}
where the matrix elements are 
\begin{subequations}\label{SUB2}
  \begin{align}
g'_{11} & \!=\!\gamma\left(\begin{array}{cc}y_{11}&y_{12}\!\\y_{21}&\!
   y_{22}\end{array}\right)^{-1}\!+\!\delta\left(\begin{array}{c}\!p_1\\\!\!p_2\!\end{array}\right)
                                                                   \left(\begin{array}{cc}p_1&p_2\end{array}\right)\!=\!
    \frac{\gamma}{D}\left(\begin{array}{cc}y_{22}&\!-y_{12}\\\!-y_{21} &y_{11}\end{array}\right)\!+\!\delta\left(\begin{array}{cc}p_1p_1&
                                                                      p_1p_2\\p_2p_1& p_2p_2\end{array}\right) ;\nonumber\\
     g'_{12} &\! =\!\gamma\left(\begin{array}{cc}y_{11}&y_{12}\\y_{21}&
                                                                  y_{22}\end{array}\!\!\right)^{-1}\left(\!\begin{array}{cc}x_{11}&x_{12}\\x_{21}&x_{22}\end{array}\!\right)\!-\!
                                                                                                                                                   \delta\left(\!\begin{array}{c}\!q_1\!\\\!q_2\!\end{array}\!\right)\left(\begin{array}{cc}p_1&p_2\end{array}
                                                                                                                                                                                                                                                 \right)\nonumber\\
    &=
      \frac{\gamma}{D}\left(\begin{array}{cc}y_{22}x_{11}-y_{12}x_{21}&
                                                                        y_{22}x_{12}-y_{12}x_{22}\\-y_{21}x_{11}+y_{11}x_{21}& -y_{21}x_{12}+y_{11}x_{22}\end{array}\right)-
                                                                                                   \delta\left(\begin{array}{cc}q_1p_1&q_1p_2\\q_2p_1&q_2p_2\end{array}\right);\nonumber\\
    g'_{13} & =-\delta\left(\begin{array}{c}p_1\\p_2\end{array}\right)\nonumber;\\
    g'_{21} &= \gamma \left(\begin{array}{cc}x_{11}& x_{12}\\x_{21}&
                                                                     x_{22}\end{array}\right)\left(\begin{array}{cc}y_{11}&y_{12}\\y_{21}&y_{22}\end{array}\right)^{-1}
                                                                                                                                           -\delta\left(\begin{array}{c}p_1\\p_2\end{array}\right)\left(\begin{array}{cc}q_1,q_2\end{array}\right);\nonumber\\
    &
      =\frac{\gamma}{D}\left(\begin{array}{cc}x_{11}y_{22}-x_{12}y_{21}
                                    & -x_{11}y_{12}+
                                  x_{12}y_{11}\\x_{21}y_{22}-x_{22}y_{21}&
                                                                           -x_{21}y_{12}+x_{22}y_{11}\end{array}\right)-\delta\left(\begin{array}{cc}p_1q_1&
                                                                                                                                                             p_1q_2
                                                                                                                                      \\p_2q_1&p_2q_2\end{array}\right);\nonumber\\
    g'_{22} &\!= \! \gamma
              \left[\left(\!\begin{array}{cc}\!y_{11}\!&\!y_{12}\!\\\!y_{21}\!&\!y_{22}\!\end{array}\!\right)\!+\!\left(\!\begin{array}{cc}x_{11}&x_{12}\\x_{21}&x_{22}\!\end{array}\!\right)
   \left(\!\begin{array}{cc}y_{11}&y_{12}\\y_{21}&y_{22}\!\end{array}\!\right)^{-1}\!\!
  \left(\!\begin{array}{cc}\!x_{11}&\!x_{12}\\\!x_{21}&\!x_{22}\!\end{array}\!\right)\!\right]\!\!+\!\! \delta\left(\!\begin{array}{c}q_1\\q_2\!\end{array}\!\right)
          \left(\!\begin{array}{cc}\!q_1\!&q_2\!\end{array}\!\right)\nonumber\\
    & =    \frac{\gamma}{D}\left(\begin{array}{cc}
                                  \begin{array}{c}
                                    (x_{11}^2+y_{11}^2)y_{22}+(x_{12}^2-y_{12}^2)y_{11}\\-2x_{11}x_{12}y_{12}\end{array}&\begin{array}{c}\!(-x_{12}^2\!-\!x_{11}x_{22}+\!\!y_{11}y_{22}-y^2_{12})y_{12}\\
                                                                                                                           +x_{11}x_{12}y_{22}\!+\!x_{12}x_{22}y_{11}\end{array}\\
                                   \begin{array}{c}
                                     \\ \end{array}\\ \begin{array}{c}\!(-x_{12}^2\!-\!x_{11}x_{22}+\!\!y_{11}y_{22}-y^2_{12})y_{12}\\
                                                                                                                           +x_{11}x_{12}y_{22}\!+\!x_{12}x_{22}y_{11}\end{array}
                                   &\begin{array}{c} (x_{12}^2-y_{12}^2)y_{22}\\\!+\!(x_{22}^2\!+\!y_{22}^2)y_{11}\!-\!2x_{12}x_{22}y_{12}\end{array}\end{array}\right)+\nonumber\\
      & ~~+\delta\! \left(\begin{array}{cc}q_1^2& q_1q_2\\q_2q_1&q_2^2\end{array}\right)
      \nonumber ;\\
    g'_{23}&= \delta \left(\begin{array}{cc}q_1
                             \\q_2\end{array}\right);~~ g'_{31}
    =-\delta \left(\begin{array}{cc}p_1& p_2\end{array}\right); ~~
    g'_{32}=\delta\left(\begin{array}{cc}q_1& q_{2}\end{array}\right);~~
    g'_{33}=\delta .\nonumber
    \end{align}
  \end{subequations}
  We calculate the determinant of the matrix $g'_{\tilde{\mc{X}}^J_{2}}$ given in \eqref{mdet} and then we find its inverse. Since  $x,~y$ are  symmetric matrices,   we write $  g'_{\tilde{\mc{X}}^J_2}(x,y,q,p,\kappa) \in MS(5,\R) $ as
\begin{equation}\label{DDD}
g'_{\tilde{\mc{X}}^J_2}(x,y,q,p,\kappa)=(g')_{i,j=1,3}=(f)_{i,j=1,5},
\end{equation}
where 
\begin{subequations}
\begin{align}
 f_{1 \le j} =
  & \begin{array}{ccccc}\begin{array}{c}\frac{\gamma}{D}y_{22}\\+\delta
                          p_1p_1\end{array}&\begin{array}{c}-\frac{\gamma}{D}y_{12}\\+\delta
                                              p_1
                                              p_2\end{array}&\begin{array}{c}\frac{\gamma}{D}(y_{22}x_{11}-y_{12}x_{12})\\-\delta
                                                               q_1p_1\end{array}&\begin{array}{c}\frac{\gamma}{D}(y_{22}x_{12}-y_{12}x_{22})\\-\delta
                                                                                   q_1p_2\end{array}&-\delta
                                                                                                      p_1\end{array}\nonumber\\
  f_{2 \le j} =
  & \begin{array}{cccc}\begin{array}{c}\frac{\gamma}{D}y_{11}\\ \delta
                         p_2
                         p_2\end{array}&\begin{array}{c}\frac{\gamma}{D}(-y_{12}x_{11}+y_{11}x_{12})\\-\delta
                                          q_2
                                          p_1\end{array}&\begin{array}{c}\frac{\gamma}{D}(-y_{12}x_{12}+y_{11}x_{22})\\-\delta
                                                           q_2
                                                           p_2\end{array}&-\delta
                                                                        p_2\end{array}\nonumber\\
  f_{3 \le j} =
  & \begin{array}{ccc} \begin{array}{c}\frac{\gamma}{D}[(x^2_{11}\!+\!y_{11}^2)y_{22}\!+\!(x^2_{12}\!-\!y^2_{12})y_{11}\!\\-\!2x_{11}x_{12}y_{12}]\!+\!\delta q^2_1\end{array}&
 \begin{array}{c}\frac{\gamma}{D}[(-x^2_{12}\!-\!x_{11}x_{22}\!+\!y_{11}y_{22}\!-\!y^2_{12}\!)\!y_{12}\\\!+\!x_{11}x_{12}y_{22}\!+\!x_{12}x_{22}y_{11}]\!+\!\delta
   q_1q_2\end{array}   &\delta  q_1  \end{array}\nonumber\\
   f_{4 \le j}
  =& \begin{array}{cc}\begin{array}{c}\frac{\gamma}{D}[(x_{12}^2-y^2_{12})y_{22}+(x^2_{22}+y^2_{22})y_{11}\\-2x_{12}x_{22}y_{12}]+\delta
                        q^2_2\end{array} &\delta
                                           q_2\end{array}\nonumber\\
  f_{55}~~=&~~~\delta\nonumber 
\end{align}
\end{subequations}

By the Laplace expansion formula \cite{lap} we have
\[
\text{det}
g'_{\tilde{\mc{X}}^J_2}(x,y,q,p,\kappa)\!=\!\sum_{i=1}^5(-1)^{1+i}f_{1i}M_{1i}\!=\!f_{11}M_{11}\!-\!f_{12}M_{12}\!+\!f_{13}M_{13}\!-\!
f_{14}M_{14}\!+\!f_{15}M_{15},
\]
where the minor $M_{ij}$ is the determinant of the $4\times 4$ - matrix obtained
from   $g'_{ij}$
removing the element $ij$.

We calculate $M_{11}$
\begin{subequations}
 \begin{align}
M_{11} &=\det \left(\begin{array}{cccc}f_{22}&f_{23}&
                                             f_{24}&f_{25}\\f_{32}&f_{33}&f_{34}&f_{35}\\f_{42}&f_{43}&f_{44}&f_{45}\\f_{52}&
                                                                                                                       f_{53}&
                                                                                                                               f_{54}&f_{55}\end{array}\right)\nonumber\\
   &=-f_{25}A\left(\begin{array}{ccc}3&4&5\\2&3&4\end{array}\right)\!+\!
                                                 f_{35}A\left(\begin{array}{ccc}2&4&5\\2&3&4\end{array}\right)\!-\!f_{45}A\left(\begin{array}{ccc}2&3&5\\2&3&4\end{array}\right)\!+\!f_{55}A\left(\begin{array}{ccc}2&3&4\\2&3&4\end{array}\right)\nonumber\end{align}\end{subequations}
 We obtain the determinant
\begin{equation}\label{DDET}
\text{det}
g'_{\tilde{\mc{X}}^J_2}(x,y,q,p,\kappa)\!=\! \delta
[\gamma^4\!-\!\gamma^3\frac{\delta I_2}{D}A\!-\!\gamma^2\frac{(\delta
  I_2)^2}{D}B\!-\!\gamma\frac{(\delta I_2)^3}{D}A\!+\!(\delta I_2)^4],
\end{equation}
where
\begin{subequations}\label{I2}\begin{align}
I_2:&\!=\!p_2q_1-p_1q_2,\quad A:=\!2[(x_{11}-x_{22})y_{12}-x_{12}(y_{11}-y_{22})],\\
  B:&\!=\! 4x_{12}^2+(x_{11}-x_{22})^2+y_{11}^2+y_{22}^2+2y^2_{12}.
  \end{align}
 \end{subequations}
Let
\begin{equation}\label{NNN}
N:=\frac{\gamma \delta}{\text{det}~(f)\times
  D}.
\end{equation}                
Using the {\it Ricci} package of {\it Wolfram Mathematica} 
system  \cite{MATH}, we get:
 \begin{Proposition}\label{PR5}
   The inverse of the matrix \eqref{DDD} is given by
\begin{subequations}
  \begin{align}
      \frac{f^{11}}{N} &\!=\!\gamma^2 \big[y_{11} (x_{12}^2\!\!-\!\!y_{12}^2)\!\!+\!\!y_{22} (x_{11}^2\!\!+\!\!y_{11}^2)\!\!-\!\!2 x_{11}x_{12}y_{12}\big]\!\!+\!\!2\gamma\delta I_2\big[\!-\!x_{12}(x_{12}^2\!\!+\!\!y_{12}^2\!\!-\!\! x_{11}x_{22}\!\!+\!\!y_{11} y_{22})\nonumber\\
    &\quad +y_{12}(x_{22}y_{11}+x_{11} y_{22})\big]\!-\!(\delta I_2)^2\big[y_{22} (x_{12}^2-y_{12}^2)+y_{11} (x_{22}^2+y_{22}^2)\!-\!2 x_{22}x_{12}y_{12}\big]\nonumber\\
     \frac{f^{12}}{N}
        &\!= \gamma^2\big[x_{12}(x_{22}y_{11}+x_{11}y_{22})-y_{12}(x_{12}^2+y_{12}^2+x_{11}x_{22}-y_{11}y_{22})\big] \nonumber\\
    &\quad +\gamma \delta I_2\big[x_{11}(x_{12}^2\!\!-\!\!y_{12}^2\!\!+\!\!x_{22}^2\!\!+\!\!y_{22}^2)\!-\!\!x_{22}(x_{12}^2\!\!-\!\!y_{12}^2\!\!+\!\!x_{11}^2\!\!+\!\!y_{11}^2)\!\!+\!\!2 x_{12}y_{12}(y_{11}\!\!-\!\!y_{22})\big]\nonumber\\
    &\quad+ (\delta I_2)^2\big[x_{12}(x_{22}y_{11}+x_{11}y_{22})-y_{12}(x_{12}^2+y_{12}^2+x_{11}x_{22}-y_{11}y_{22})\big]\nonumber\\
    \frac{f^{13}}{N}&=\gamma^2(x_{12}y_{12}\!\!-\!\!x_{11}y_{22})\!-\!\gamma\delta
                      I_2[x_{12}(x_{11}\!+\!x_{22})\!+\!y_{12}(y_{11}\!\!+\!\!y_{22})]\!+\!(\delta
                      I_2)^2(x_{22}y_{11}\!\!-\!\!x_{12}y_{12})\nonumber\\
    \frac{f^{14}}{N}&=\gamma^2(-x_{12}y_{11}+x_{11}y_{12})+\gamma\delta
                      I_2(x^2_{11}+2x^2_{12}-x_{11}x_{22}+y^2_{11}+y^2_{12})\nonumber\\
                     &\quad +(\delta
                       I_2)^2(2x_{11}y_{12}-2x_{12}y_{11}-x_{22}y_{12}+x_{12}y_{22})
                       -\frac{D(\delta I_2)^3}{\gamma}\nonumber\\
    \frac{f^{15}}{N}
      &=\!\gamma^2\big[q_1(x_{11}y_{22}\!-\!x_{12}y_{12})\!+\!q_2(x_{12}y_{11}\!-\!x_{11}y_{12})\!+\! p_1\big(y_{11}(x_{12}^2\!-\!y_{12}^2)\!+\!y_{22}(x_{11}^2\!+\!y_{11}^2)\nonumber\\
      &\quad -2x_{11}x_{12} y_{12}\big)\!+\! p_2\big(\!-\!y_{12}(x_{12}^2\!+\!y_{12}^2\!+\!x_{11}x_{22}\!-\!y_{11}y_{22})\!+\!x_{12}(x_{11}y_{22}\!+\!x_{22}y_{11})\big)\big]\nonumber\\
      &\quad +\gamma\delta I_2 \Big[q_1\big(x_{12}(x_{11}\!+\!x_{22})\!+\!y_{12}(y_{11}\!+\!y_{22})\big)
      \!-\!q_2(2 x_{12}^2\!-\!x_{11}x_{22}\!+\!x_{11}^2\!+\!y_{11}^2\!+\!y_{12}^2)
      \nonumber\\
      &\quad+2p_1\big( x_{12}(x_{11}x_{22}-y_{11}y_{22}-x_{12}^2-y_{12}^2)+y_{12}(x_{22}y_{11}+x_{11}y_{22})\big)
      \nonumber\\
      &\quad+p_2\big(x_{11}(x_{22}^2\!\!+\!\!y_{22}^2\!\!+\!\!x_{12}^2\!\!-\!\!y_{12}^2)\!\!-\!\!x_{22}(x_{11}^2\!\!+\!\!y_{11}^2\!\!+\!\!x_{12}^2\!\!-\!\!y_{12}^2)\!\!+\!\!2x_{12}y_{12}(y_{11}\!\!-\!\!y_{22})\big)\Big]\nonumber\\
      &\quad+(\delta I_2)^2\Big[q_1(x_{12}y_{12}-x_{22}y_{11})+q_2(2x_{12}y_{11}-2x_{11}y_{12}+x_{22}y_{12}-x_{12}y_{22})\nonumber\\
      &\quad +p_2 \big( x_{12}(x_{22}y_{11}+x_{11}y_{22})-y_{12}(y_{12}^2+x_{12}^2+x_{11}x_{22}-y_{11}y_{22})\big)  \nonumber\\
      &\quad - p_1\big(y_{11}(x_{22}^2+y_{22}^2)+ y_{22}(x_{12}^2-y_{12}^2)-2x_{12}x_{22}y_{12}\big)\Big]+\frac{D (\delta I_2)^3 }{\gamma} q_2\nonumber\\
    \frac{f^{22}}{N}&=\gamma^2\big[y_{11}(x^2_{22}\!+\!y^2_{22})\!+\!y_{22}(x^2_{12}\!-\!y^2_{12})\!-\!2x_{22}x_{12}y_{12}\big]
                      \!+\!2\gamma \delta
                      I_2\big[-y_{12}(x_{11}y_{22}\!+\!x_{22}y_{11})\nonumber
                      \\ &\quad +x_{12}(x^2_{12}\!\!+\!\!y^2_{12}\!\!-\!\!x_{11}x_{22}\!\!+\!\!y_{11}y_{22})\big]
                    \!\!-\!\!(\delta  I_2)^2\big[y_{22}(x^2_{11}\!\!+\!\!y^2_{11})\!\!+\!\!y_{11}(x^2_{12}\!\!-\!\!y^2_{12})\!\!-\!\!2x_{11}x_{12}y_{12}\big]\nonumber\\
    \frac{f^{23}}{N}&=\gamma^2(x_{22}y_{12}-x_{12}y_{22})+\gamma
                      \delta
                      I_2(-2x_{12}^2+x_{11}x_{22}-x^2_{22}-y^2_{12}-y^2_{22})\nonumber\\
                     & \quad +(\delta
                      I_2)^2(x_{12}y_{11}-x_{11}y_{12}+2x_{22}y_{12}-2x_{12}y_{22})+\frac{D(\delta
                       I_2)^3}{\gamma}\nonumber\\
 \frac{f^{24}}{N}&\!=\!-\gamma^2(x_{22}y_{11}\!-\!x_{12}y_{12})\!\!+\!\!\gamma \delta
                   I_2[x_{12}(x_{11}\!\!+\!\!x_{22})\!\!+\!\!y_{12}(y_{11}\!\!+\!\!!y_{22})]\!-\!(\delta
                   I_2)^2(x_{12}y_{12}\!\!-\!\!x_{11}y_{22})\nonumber\\
    \frac{f^{25}}{N}&=\gamma^2\Big[q_1(x_{12}y_{22}-x_{22}y_{12})
    +q_2(x_{22}y_{11}-x_{12}y_{12})
    +p_1\big[x_{12}(x_{11}y_{22}+x_{22}y_{11})\nonumber\\
   & \quad-y_{12}(y_{12}^2\!\!+\!\!x_{12}^2\!\!+\!\!x_{11}x_{22}\!-\!y_{11}y_{22})\big]
  \!\!+\!\!p_2\big[y_{11}(x_{22}^2\!\!+\!\!y_{22}^2)\!\!+\!\!y_{22}(x_{12}^2\!-\!y_{12}^2)\!\!-\!\!2x_{12}x_{22}y_{12} \big]\Big]  \nonumber\\
  &\quad-\gamma \delta I_2\Big[\!-\!q_1(2x_{12}^2\!\!-\!\!x_{11}x_{22}\!\!+\!\!x_{22}^2\!\!+\!\!y_{12}^2\!\!+\!\!y_{22}^2)\!\!+\!\!q_2\big(x_{12}(x_{11}\!\!+\!\!x_{22})\!\!+\!\!y_{12}(y_{11}\!\!+\!\!y_{22})\big)\nonumber\\
  &\quad +2p_2\big[x_{12}(x_{11}x_{22}-x_{12}^2-y_{12}^2-y_{11}y_{22})+y_{12}(x_{22}y_{11}+x_{11}y_{22})\big]\nonumber\\
  &\quad +p_1\big[\!-\!x_{11}(x_{12}^2\!\!+\!\!x_{22}^2\!\!+\!\!y_{22}^2\!\!-\!\!y_{12}^2)\!\!+\!\!x_{22}(x_{11}^2\!\!+\!\!x_{12}^2\!\!+\!\!y_{11}^2\!\!-\!\!y_{12}^2)\!\!-\!\!2x_{12}y_{12}(y_{11}\!\!-\!\!y_{22})\big]\Big]\nonumber\\
  &\quad-(\delta I_2)^2\Big[q_1(x_{12}y_{11}-x_{11}y_{12}+2x_{22}y_{12}-2x_{12}y_{22}) +q_2(x_{11}y_{22}-x_{12}y_{12})\nonumber\\
   &\quad +p_1\big[-x_{12}(x_{11}y_{22}+x_{22}y_{11})
   +y_{12}(y_{12}^2+x_{12}^2+x_{11}x_{22}-y_{11}y_{22})\big]\nonumber\\
  &\quad +p_2\big[y_{11}(x_{12}^2\!\!-\!\!y_{12}^2)\!\!+\!\!y_{22}(x_{11}^2\!\!+\!\!y_{11}^2)\!\!-\!\!2x_{11}x_{12}y_{12} \big]\Big]
                  -\frac{D(\delta I_2)^3}{\gamma}q_1     \nonumber\\
    \frac{f^{33}}{N}& \!=\!\gamma^2y_{22}\!+\!2\gamma\delta I_2x_{12}\!-\!(\delta
                       I_2)^2y_{11}\nonumber\\
   \frac{ f^{34}}{N}&=-\gamma^2y_{12}-\gamma\delta
                      I_2(x_{11}-x_{22})-(\delta
                      I_2)^2y_{12}\nonumber\\
  \frac{ f^{35}}{N}
                     &=\gamma^2\big[-q_1 y_{22}+q_2y_{12}-p_1(x_{11}y_{22}-x_{12}y_{12})+p_2(x_{22}y_{12}-x_{12}y_{22})\big]
                     \nonumber\\
                     &\quad - \gamma\delta I_2\Big[ 2q_1 x_{12}-q_2(x_{11}-x_{22})+p_1\big(x_{12}(x_{11}+x_{22})+y_{12}(y_{11}+y_{22})\nonumber\\&\quad+p_2(2x_{12}^2-x_{11}x_{22}+x_{22}^2+y_{12}^2+y_{22}^2)\big)\Big]+(\delta I_2)^2\Big[q_1y_{11}+q_2y_{12}\nonumber\\
                     &\quad +p_1(x_{22}y_{11}-x_{12}y_{12})+p_2(x_{12}y_{11}-x_{11}y_{12}+2x_{22}y_{12}-2x_{12}y_{22})\Big] +\frac{D(\delta I_2)^3}{\gamma}p_2\nonumber\\
    \frac{ f^{44}}{N}&=\gamma^2y_{11}-2\gamma\delta I_2 x_{12}-(\delta
                       I_2)^2y_{22}\nonumber\\
    \frac{
    f^{45}}{N}&=\gamma^2\Big[q_1y_{12}-q_2y_{11}+p_1(x_{11}y_{12}-x_{12}y_{11})+p_2(x_{12}y_{12}-x_{22}y_{11})\Big]\nonumber\\
    &\quad+\gamma\delta I_2\Big[q_1(x_{11}-x_{22})+2q_2x_{12}+p_1(x_{11}^2+y_{11}^2+y_{12}^2 +2x_{12}^2-x_{11}x_{22})\nonumber\\
    &\quad+p_2\big(x_{12}(x_{11}+x_{22})+y_{12}(y_{11}+y_{22})\big)\Big]
    +(\delta I_2)^2\Big[q_1y_{12}+q_2y_{22}\nonumber\\
    &\quad+p_1(x_{12}y_{22}-x_{22}y_{12}+2x_{11}y_{12}-2x_{12}y_{11})+p_2(x_{11}y_{22}-x_{12}y_{12})\Big]-\frac{D(\delta I_2)^3}{\gamma}p_1  \nonumber\\
\frac{ f^{55}}{N}&= \frac{D\gamma^3 }{\delta}\!\!+\!\!\gamma^2\Big[q_1^2 y_{22}\!+\!q_2^2 y_{11}\!-\!2q_1 q_2 y_{12}\!\!+\!\!p_1^2\big[(x_{12}^2\!-\!y_{12}^2)y_{11}\!\!+\!\!(x_{11}^2\!\!+\!\!y_{11}^2)y_{22}\!\!-\!\!2x_{11}x_{12}y_{12}\big]\nonumber\\
&\quad+p_2^2\big[(x_{22}^2+y_{22}^2)y_{11} +(x_{12}^2-y_{12}^2)y_{22}-2x_{22}x_{12}y_{12}\big]+2p_1 p_2\big[x_{12}(x_{11}y_{22}+x_{22}y_{11})\nonumber\\
&\quad-(x_{12}^2\!\!+\!\!y_{12}^2\!\!+\!\!x_{11}x_{22}\!\!-\!\!y_{11}y_{22})y_{12}\big]\!\!+\!\!2 p_1 q_1 (x_{11}y_{22}\!\!-\!\!x_{12}y_{12})\!\!+\!\! p_1 q_2 (x_{12}y_{22}\!\!-\!\!x_{22}y_{12})\nonumber\\
&\quad+2p_2 q_1(x_{12}y_{11}\!\!-\!\!x_{11}y_{12})\!\!+\!\!2 p_2 q_2
 (x_{22}y_{11}\!\!-\!\!x_{12}y_{12})\Big]\!\!+
 \!\!2(\gamma\delta I_2)\Big[(q_1^2\!\!-\!\!q_2^2) x_{12}\!\!-\!\!q_1q_2(x_{11}\!\!-\!\!x_{22})
\nonumber\\
&\quad\!-\!(p_1^2-p_2^2)\big[x_{12}(x_{12}^2+y_{12}^2-x_{11}x_{22}+y_{11}y_{22})\!-\!y_{12}(x_{22}y_{11}+x_{11}y_{22})\big]\nonumber\\
&\quad+(p_1q_1\!-\!p_2q_2)\big[x_{12}(x_{11}\!\!+\!\!x_{22})\!\!+\!\!y_{12}(y_{11}\!\!+\!\!y_{22})\big]\!\!-\!\!\frac{1}{2}(p_1q_2\!\!+\!\!p_2q_1)(x_{11}^2\!\!-\!\!x_{22}^2\!\!+\!\!y_{11}^2\!\!-\!\!y_{22}^2)\nonumber\\
&\quad+p_1p_2\big[x_{11}(x_{12}^2\!\!+\!\!x_{22}^2\!\!+\!\!y_{22}^2\!\!-\!\!y_{12}^2)\!\!-\!\!x_{22}(x_{12}^2\!\!+\!\!x_{11}^2\!\!+\!\!y_{11}^2\!\!-\!\!y_{12}^2)\!\!+\!\!2x_{12}y_{12}(y_{11}\!\!-\!\!y_{22})\big]\Big]\nonumber\\
&\quad-(\delta I_2)^2\Big[q_1^2y_{11}\!\!+\!\!q_2^2y_{22}\!+\!2q_1q_2y_{12}\!\!+\!\!p_1^2\big[(x_{12}^2\!\!-\!\!y_{12}^2)y_{22}\!\!+\!\!(x_{22}^2\!\!+\!\!y_{22}^2)y_{11}\!\!-\!\!2x_{12}x_{22}y_{12}\big]\nonumber\\
&\quad+p_2^2\big[(x_{12}^2-y_{12}^2)y_{11}+(x_{11}^2+y_{11}^2)y_{22}-2x_{12}x_{11}y_{12}\big]-2p_1q_1(x_{12}y_{12}-x_{22}y_{11})\nonumber\\
&\quad-2p_1q_2(x_{12}y_{11}-x_{11}y_{12})-2p_2q_1(x_{12}y_{22}-x_{22}y_{12})-2p_2q_2(x_{12}y_{12}-x_{11}y_{22})\nonumber\\
&\quad-2p_1p_2\big[x_{12}(x_{11}y_{22}+x_{22}y_{11})- y_{12}(x_{12}^2+y_{12}^2+x_{11}x_{22}-y_{11}y_{22})\big]\Big]\!\!-\! \frac{D(\delta I_2)^3 }{\gamma}I_2\nonumber
    \nonumber 
  \end{align}
   \end{subequations}
   

\end{Proposition}
We   systematize  the results of Proposition \ref{PR5}:
\begin{Remark}\label{R1}
The elements of the inverse matrix of \eqref{DDD} can be written in the following condensed form:
  \[
  f^{ij}=N\Big[\frac{D\gamma^3}{\delta}P_0+\gamma^2P_1+\gamma \delta I_2 P_2+(\delta I_2)^2 P_3+\frac{D(\delta I_2)^3}{\gamma}P_4\Big], 
    \]
  
\noindent where $N$ is defined in \eqref{NNN} and $ P_0,P_1,P_2,P_3,P_4$ are polynomials in $x\!=\!x_{ij}$ and $y\!=\!y_{ij}$. We have that $P_0=0$ for all $f^{ij}$, excepting $f^{55}$, for which $P_0=1$. Similarly, $P_4\neq 0$ only for the elements $f^{i5}$, $\forall i=1,...,5$.
\end{Remark}
\begin{Remark}\label{R2} The inverse matrix \eqref{DDD} is equivalent with the inverse matrix given in \eqref{316}. To check the consistency of our calculations, we computed some of the matrix elements with the two methods: the one in Proposition \ref{PR2} and the one in Proposition \ref{PR5}. For example, one can easily check that the element $f^{13}$ is equal with the element on line 1 and column 3 of $b_{11}$.
\end{Remark}

\subsection*{Acknowledgements}

This research was conducted in the framework of the
ANCS project programs PN 19 06 01 01/2019 and PN 23 21 01 01/2023.
We would like to thank  Marian Apostol for interest and suggestions.


\begin{thebibliography}{99}

\footnotesize\itemsep=0pt

\bibitem{bs}R. Berndt,  R. Schmidt, {\it Elements of the representation
theory of the Jacobi group}, Progress in Mathematics,  Vol.  163,  Birkh\"auser
Verlag, Basel, 1998. 

 \bibitem{ez}M. Eichler,  D. Zagier, {\it The theory of Jacobi forms},
 Progress in
 Mathematics{\bf 55}, Birkh\"{a}user Boston, Inc., Boston, MA, 1985.
 
\bibitem{tak}K. Takase, {\it  A note on automorphic forms}, J. Reine
  Angew. Math. {\bf 409} (1990) 138 -- 171.

 \bibitem{helg} S. Helgason, \emph{Differential geometry, Lie groups and
symmetric spaces}, {Academic Press, New York, 1978}.

\bibitem{nou}
S. Berceanu,{\it{ A convenient coordinatization of {S}iegel--{J}acobi domains}},
 Rev. Math. Phys.
 \textbf{24} (2012) 1250024, 38~pages; arXiv:1204.5610.
 
  \bibitem{SB21}S.  Berceanu,   {\it Geodesics on the extended Siegel-Jacobi upper
half-plane}, Romanian J. Phys.  {\bf 66}  (2021) 107,   28 pages;
arXiv:2101.08015[math.DG].

\bibitem{SB15}S. Berceanu, {\it Balanced metric and Berezin quantization on  the
Siegel--Jacobi ball}, SIGMA {\bf 12} (2016) 064, 24 pages; arXiv:1512.00601.

\bibitem{SB16}S. Berceanu,  Remarks on Berezin quantization on the
  Siegel-Jacobi ball, in  {\it Springer Group 31 Proceedings, Physical
    and Mathematical Aspects of Symmetries}, Editors
  S. Duarte, J.P. Gazeau, S. Faci, T. Micklitz,
  R. Scherer, F. Toppan, 105--110, 2016.

\bibitem{SB17}S. Berceanu, Coherent states associated to the Jacobi group  
  and Berezin quantization of  the Siegel-Jacobi ball, in {\it Trends in Mathematics}, Springer,
 Birkhäuser Basel, Editors  Piotr Kielanowski,  Anatol Odzijewicz,
 Emma Previato, (2019) 31-- 37.


 \bibitem{SB20} E.M. Babalic, S. Berceanu, {\it  Remarks on the geometry of the extended Siegel--Jacobi upper half-plane}, Romanian J. Phys. {\bf 65}   (2020) 113, 27 pages;  arXiv:2002.04452 
 
\bibitem{SB20N} S.  Berceanu, {\it  Invariant metric on the extended Siegel-Jacobi
 upper half space}, 
J. Geom. Phys. {\bf 162} (2021) 104049, 20 pages;
arXiv:math.DG/2006.03319.

\bibitem{ber73}F.A. Berezin, {\it Quantization in complex
bounded domains} (Russian),  Dokladi Akad. Nauk SSSR, Ser. Mat. {\bf
211} (1973) 1263--1266.


\bibitem{ber74}F.A. Berezin, {\it Quantization}  (Russian),
Izv. Akad. Nauk SSSR Ser. Mat. {\bf 38}  (1974) 1116--1175.

\bibitem{ber75}F.A. Berezin, {\it Quantization in complex symmetric
  spaces}  (Russian), Izv. Akad. Nauk SSSR Ser. Mat. {\bf 39} (1975) 363--402.

\bibitem{berezin}F.A. Berezin, {\it The general concept of quantization},
  Commun. Math. Phys. {\bf 40} (1975) 153--174.
  
\bibitem{arr}
C. Arezzo, A. Loi, {\it  Moment maps, scalar curvature and quantization of
 {K}\"{a}hler manifolds}, {Comm. Math. Phys.} \textbf{246} (2004)
 543 -- 559.
  
  
  \bibitem{alo}
  A. Loi,  R. Mossa,  {\it  Berezin quantization of homogeneous bounded domains},
{Geom. Dedicata} \textbf{161} (2012)  119 -- 128.
  

\bibitem{SB03}
S. Berceanu, Realization of coherent state {L}ie algebras by differential
 operators, in {\it Advances in Operator Algebras and Mathematical Physics},
 {Theta Ser. Adv. Math.}, Vol.~5, Theta, Bucharest, 2005, 1 -- 24; 
 {arXiv: 0504053/math.DG}.
 
 \bibitem{SB14}S. Berceanu, {\it Coherent states and geometry on the Siegel--Jacobi
disk}, Int. J. Geom. Methods Mod. Phys.  {\bf{11} }
(2014) 1450035, 25 pages,  arXiv:1307.4219.

  
\bibitem{perG}A. M. Perelomov,  {\it Generalized Coherent States and their
Applications}, Texts and
 Monographs in Physics, Springer-Verlag, Berlin, 1986.

 \bibitem{GH}P. Griffiths, J. Harris,  {\it Principles of algebraic
    geometry}, {\rm John Wiley and Sons, New York - Chichester - 
    Brisbane - Toronto,  1978}.
    
 \bibitem{lutke}H. L\"utkepohl, {\it Handbook of matrices},  John Wiley
  \& Sons,  Chikester, 1996.

 
\bibitem{bern84}
R. Berndt, Some differential operators in the theory of {J}acobi forms,
 preprint IHES/M/84/10, 1984, 31~pages.

 \bibitem{cal3}
 E. K\"{a}hler, Raum-{Z}eit-{I}ndividuum, \textit{Rend. Accad. Naz. Sci. XL Mem.
 Mat.} \textbf{16} (1992) 115 -- 177.

\bibitem{SB19}S.  Berceanu, {\it The real Jacobi group
    revisited},  SIGMA {\bf 15} (2019) 096, 50 pages; arXiv:1903.1072
  [math.DG], v1, 93 pages;  v2, 54 pages.


 \bibitem{ccl} K.S. Chern, S.S. Chen, W.H. Lam, {\it Lectures on
     Differential Geometry}, World Scientific,   Series of University
   Mathematics, Vol 1,  Singapore, New Jersey, London, Hong Kong, 2000,
   356 pages.
   
   \bibitem{SB22}S.  Berceanu, {\it  Hamiltonian systems on almost cosymplectic
manifolds}, J. Geom.  Phys.  {\bf 183} (2023) 10470026, 20
pages;   arXiv:2201.01962 [math.DG], 26 pages.

\bibitem{SB24} S. Berceanu,  {\it  Berry phases and connection matrices defined  on
 homogeneous spaces attached to Siegel-Jacobi groups},
arXiv:2401.03758v1 [math.DG].
 

\bibitem{kn1} S. Kobayashi, K. Nomizu, {\it Foundations of Differential
Geometry}, {V}ol.~{I},
 Interscience Publishers, New York - London, 1963.

\bibitem{lap}Laplace expansion, wikipedia,
  https://en.wikipedia.org/wiki/Laplace expansion.
  
\bibitem{MATH}Mathematica version 12, https://www.wolfram.com/mathematica/.


  
\end{thebibliography}
\end{document}